\documentclass[12pt]{amsart}

\usepackage{amssymb,latexsym}
\usepackage{amsmath,amsfonts,amsthm,amscd,amsxtra,enumerate}
\usepackage{mathtools}
\usepackage[all]{xy}
\usepackage{hyperref}
\usepackage{amsmath,latexsym,amssymb, enumerate, amsthm, epsfig }
\usepackage{bookmark}
\usepackage[margin=1in]{geometry}
\usepackage{epsfig}
\usepackage{graphicx}
\usepackage{enumerate}

\newtheorem{con}{Conjecture}[section]
\newtheorem{defn}[con]{Definition}

\newtheorem{thm}[con]{Theorem}
\newtheorem*{thm*}{Theorem}
\newtheorem{lemma}[con]{Lemma}

\newtheorem{corollary}[con]{Corollary}
\newtheorem{proposition}[con]{Proposition}

\newtheorem*{con*}{Conjecture}

\AtBeginDocument{}
\begin{document}
	\title{ Some results on the Ryser design conjecture-II}
	\author[Tushar D. Parulekar]{Tushar D. Parulekar}
	\address{Department of Mathematics, Indian Institute of Technology Bombay, Powai, Mumbai 400076}
	\email{tushar.p@math.iitb.ac.in}
	\author[Sharad S. Sane]{Sharad S. Sane}
	\address{Chennai Mathematical Institute, SIPCOT IT Park, Siruseri, Chennai 603103}
	\email{ssane@cmi.ac.in}
	\date{\today}
	
	
	
	\begin {abstract}
	A Ryser design $\mathcal{D}$ on $v$ points is a collection of $v$ proper subsets (called blocks) of a point-set with $v$ points satisfying (i) every two blocks intersect each other in $\lambda$ points for a fixed $\lambda < v$  (ii) there are at least two block sizes. A design $\mathcal{D}$ is called a symmetric design, if all the blocks of $\mathcal{D}$ have the same size
	(or equivalently, every point  has the same replication number) and every two blocks intersect each other in $\lambda$ points. The only known construction of a Ryser design is via block complementation of a symmetric design also known as the Ryser-Woodall complementation method. Such a Ryser design is called a Ryser design of  Type-1. The Ryser-Woodall conjecture states: ``every Ryser design is of Type-1". 
	Main results of the present article are the following. An expression for the inverse of the incidence matrix $\mathsf{A}$ of a Ryser  design is obtained. A necessary condition for the design to be of Type-1 is obtained. A well known conjecture states that, for a Ryser design on  \textit{v}  points $\mbox{ }4\lambda-1\leq v\leq\lambda^2+\lambda+1$.
	A partial support for this conjecture is obtained. Finally, a special case of Ryser designs with two block sizes is shown to be of Type-1.
\end{abstract}
\maketitle
\section{Introduction}
\noindent

Let $X$ be a finite set of points and $L\subseteq P(X)$, where $P(X)$ is the power set of $X$. Then the pair $(X,L)$ is called a design. The elements of $X$ are called its points and the members of $L$ are called the blocks. Most of the definitions, formulas and proofs of standard results used here can be found in \cite{Ion2}.
\begin{defn}{\rm
		A design $\mathcal{D}=(X,L)$ is said to be a \textit{symmetric} $(v,k,\lambda)$ design  if
		\begin{enumerate}[{\rm 1.}]
			\item $ |X|=|L|=v$,
			\item $ |B_1\cap B_2|=\lambda \geq 1$ for all blocks $B_1$ and $B_2$ of $\mathcal{D},~~ B_1\neq B_2$,
			\item $ |B|=k>\lambda$ for all blocks $B$ of $\mathcal{D}$.
		\end{enumerate}
}\end{defn}
\begin{defn}{\rm
		A design $\mathcal{D}=(X,L)$ is said to be a \textit{Ryser design} of order $v$ and index $\lambda$ if it satisfies the following conditions:
		\begin{enumerate}[{\rm 1.}]
			\item $|X|=|L|=v$,
			\item $ |B_1\cap B_2|=\lambda$ for all blocks $B_1$ and $B_2$ of $\mathcal{D}, B_1\neq B_2$,
			\item $ |B|>\lambda$ for all blocks $B$ of $\mathcal{D}$,
			\item there exist blocks $B_1$ and $B_2$ of $\mathcal{D}$ with $|B_1|\neq|B_2|$.
		\end{enumerate}
}\end{defn}
\noindent
Let $\mathcal{D}=(X,\mathcal{A})$ be a symmetric $(v,k,k-\lambda)$ design with $k\neq 2\lambda$. Let $A$ be a fixed block of $\mathcal{D}$.
Form the collection $\,\mathcal{B}=\{A\}\bigcup \{A\triangle B: B\in \mathcal{A}, B\neq A\}$, where $A\triangle B$ denotes the usual symmetric difference of $A$ and $B$. Then $\overline{\mathcal{D}}=(X,\mathcal{B})$ is a Ryser design of order $v$ and index $\lambda$ obtained from $\mathcal{D}$ by block complementation with respect to the block $A$. We denote $\overline{\mathcal{D}}$ by $\mathcal{D}*A$. 
Then $A$ is also a block of $\mathcal{D}*A$ and the original design $\mathcal{D}$ can be obtained by complementing $\mathcal{D}*A$ with respect to the block $A$. 
If $\mathcal{D}$ is a symmetric $(v,k,\lambda^{'})$ design, then the design obtained by complementing $\mathcal{D}$ with respect to some block is a Ryser design of order $v$ with index $\lambda=k-\lambda^{'}$. A Ryser design obtained in this way is said to be of \textit{Type-1}.\\
All the known examples of Ryser designs can be described by the above construction that was given by Ryser.  This construction is also called a Ryser-Woodall complementation or simply a block complementation.\\
Define a Ryser design to be of \textit{Type-2} if it is not of Type-1. The Ryser Design Conjecture states \textit{``Every Ryser design is of Type-1''}.
The conjecture has been proved to be true for various values of $\lambda$ and $v$. 
Ryser and Woodall independently proved the following result:
\begin{thm}[{\cite[Theorem 14.1.2]{Ion2}}{ Ryser-Woodall Theorem}]\label{thm:RyserWoodall}
	If $\mathcal{D}$ is a Ryser design of order $v$, then there exist integers $r_1$ and $r_2$, $r_1\neq r_2$ such that $r_1+r_2=v+1$ and any point occurs either in $r_1$ blocks or in $r_2$ blocks.
\end{thm}
\noindent
Let $\mathcal{D}$ be a Ryser design of order $v$ and index $\lambda$. It is known that a Ryser design has two replication numbers $r_1 > r_2$ with $r_1 + r_2 = v + 1$ such that every point is in either $r_1$ blocks or $r_2$ blocks.  
Following Singhi and Shrikhande {\cite{Sin}} we define $\rho=(r_1-1)/(r_2-1)=c/d$, where $\gcd(c,d)=1.$ 
Let $g=\gcd(r_1-1,r_2-1). \text{ 
	Then } r_1+r_2=v+1 \text{ implies } g \text{ divides } (v-1)$,  
$~~r_1-1=cg,~ r_2-1=dg \text{ and } v-1=(c+d)g$.
We also write $a$ to denote $c - d$ and observe that any two of $c, d$ and $a$ are coprime to each other. 
The point-set is partitioned into subsets $E_1$ and $E_2$, where $E_i$ is the set of points with replication number $r_i$ and let $e_i=|E_i|$ for $ i = 1, 2$. Then $e_1,e_2 > 0$ and $e_1 + e_2 = v$.
For a block $A$, let us denote $|E_i\cap A|$, the number of points of block $A$ with replication number $r_i$ by $\tau_i(A)$, for $i=1,2$.
Then $|A|=\tau_1(A)+\tau_2(A)$. \textit{We say a block $A$ is large, average or small depending on whether $|A|$ is greater than $2\lambda$, equal to $2\lambda$ or less than $2\lambda$ respectively.} 
The Ryser-Woodall complementation of a Ryser design $\mathcal{D}$ of index $\lambda$ with respect to some block $\,A \in \mathcal{D}\,$ is either a symmetric design or a Ryser design of index $(|A|-\lambda)$. If $\mathcal{D}*A$ is the new Ryser design of index $(|A|-\lambda)$ obtained by Ryser-Woodall complementation of a Ryser design $\mathcal{D}$ with respect to the block $A$, we denote the new parameters of $\mathcal{\mathcal{D}*A}$ by $\lambda(\mathcal{D}*A), r_1(\mathcal{D}*A)$ etc.\\
\begin{proposition}[{\cite[Proposition 14.1.7]{Ion2}}] \label{prop:complement-properties}
	Let $\mathcal{D}$ be a Ryser design of Type-2 and let $A$ be a block of 
	$\mathcal{D}$. Then $\mathcal{D}* A$ is a Ryser design with
	$r_1(\mathcal{D}*A)=r_1(\mathcal{D})$ and $\lambda(\mathcal{D}*A)=|A|-\lambda(\mathcal{D})$.
\end{proposition}

\begin{thm}[{\cite[Theorem 14.1.17]{Ion2}}] 
	\label{thm:rho<lambda}
	For any Ryser design with block intersection $\lambda>1$ and replication numbers $r_1 \text{ and } r_2, 
	\quad\lambda/(\lambda-1)\leq \rho \leq \lambda~$ and  $~\rho \notin (\lambda-1,\lambda),~$ where $\qquad\rho=(r_1-1)/(r_2-1)$.
\end{thm}    
\noindent
Ionin and Shrikhande \cite{Ion1} made the following conjecture.
\begin{con}\label{con:bound_on_v}
	For any Ryser design on  \textit{v}  points $\mbox{ }4\lambda-1\leq v\leq\lambda^2+\lambda+1$.
\end{con}
\begin{thm}[{\cite[Theorem 9]{Kram}}] \label{Thm:Type-1_iff_2columnsum_1occurs_exactly_once}
	A Ryser design $\mathcal{D}$ is of Type-1 if and only if $\mathcal{D}$ has two column sums one of which occurs exactly once.
\end{thm}
\noindent
In \cite{Ser2} Seress introduced the term $D=e_1-r_2=r_1-e_2-1$ and proved the following result.
\begin{thm}\label{thm:D=0D=-1}
	A Ryser design is of Type-1 if and only if $~~ D=0 ~~\text{ or }~~ D=-1.$
\end{thm}
\noindent
We use the following equations which can be found in \cite{Sin} and \cite{Ion1}.
In a Ryser design with block sizes $ k_1,k_2,\ldots,k_v$  
\begin{equation}\label{sum_of_all}
\sum_{m=1}^{v}\frac{1}{k_m-\lambda}=\frac{(\rho +1)^2}{\rho}-\frac{1}{\lambda}
\end{equation}
\begin{align}
&(\rho-1)e_1=\lambda(\rho + 1)-r_2 \label{e1form}\\
&e_1=\lambda + \frac{\lambda+D}{\rho} \label{e1Dform}\\
&(\rho-1)e_2=\rho r_1-\lambda(\rho + 1) \label{e2form}\\
\intertext{and}
&e_2= \lambda +[\lambda -(D+1)]\rho \label{e2Dform}.\\
\intertext{From Equations (\ref{e2form}) and (\ref{e1form}) we get,}
&r_1 = 2\lambda + \left(\frac{a}{c}\right)(e_2-\lambda) \label{r1form}\\
&r_2=2\lambda-\left(\frac{a}{d}\right)(e_1-\lambda) \label{r2form}
\end{align}
\noindent
\begin{equation}\label{e1rhopluse2byrho_this_ch}
1 +\rho e_1+\dfrac{e_2}{\rho}= \lambda\dfrac{(\rho +1)^2}{\rho}.
\end{equation}
Using a simple two way counting we get,
\begin{equation}\label{tau1r1-1+tau2r2-1}
(r_1-1)\tau_1(A)+(r_2-1)\tau_2(A)=\lambda(v-1)
\end{equation}
which implies
\begin{equation}\label{tau1rhotau2form}
\rho\tau_1(A)+\tau_2(A)=\lambda (\rho + 1).
\end{equation}
\noindent
After dividing Equation (\ref{tau1r1-1+tau2r2-1}) by $g$, the common gcd of  $r_1 -1, r_2 - 1$ and $v - 1$ and using the coprimality of $c$ and $d$ we get
\begin{align*}
&\tau_1(A)=\lambda-td\\
&\tau_2(A)=\lambda+tc\\
&|A|=2\lambda+ta
\end{align*}
for some integer $t$. Hence we get the following lemma:
\begin{lemma}\label{lemma:blocksize}
	Let $A$ be any block of a Ryser design. Then the size of $A$ has the form  $|A|=2\lambda+ta$, where $t$ is an integer. The block $A$ is large, average or small depending on whether $t>0,  t=0$ or $t<0$ respectively. Hence $\tau_1(A)=\tau_2(A)=\lambda$ if $A$ is an average block,  $\tau_1(A)>\lambda>\tau_2(A)$ if $A$ is a small block and  $\tau_2(A)>\lambda>\tau_1(A)$ if $A$ is a large block.
\end{lemma}
\noindent
Let $x=(e_2-\lambda)/c$ in Equation (\ref{r1form}). Then, $r_1 = 2\lambda +xa$.
Since $c$ and $a$ are co-prime, it is clear that $c$ divides $e_2 - \lambda$ and hence $x$ is an integer. Therefore we get
\begin{equation}\label{e2lambdaxcform}
e_2=\lambda+xc
\end{equation}
Similarly let $y=(e_1-\lambda)/d$ in Equation (\ref{r2form}). Then, $r_2 = 2\lambda - ya$, where $y$ is an integer and
\begin{equation*}
e_1=\lambda+yd
\end{equation*}
\noindent
In this article, we prove the following results. 
\begin{thm}\label{Inverse of incidence matrix}
	Let $\mathsf{A}$ be the incidence matrix of a Ryser design of order $v$ and index $\lambda$ with block sizes $k_i,~~ i=1,2,\ldots,v$. Let 
	\begin{equation}\label{MatrixD}
	\mathsf{D}=diag(k_1-\lambda,k_2-\lambda,\ldots,k_v-\lambda)
	\end{equation}
	Let 
	\begin{equation}\label{MatrixR}
	\mathsf{R}=
	\left(
	\begin{array}{ccc}
	\rho \mathsf{J}_{e_1\times e_1} & \vline & \mathsf{J}_{e_1\times e_2}  \\
	\hline & \vline &      \\
	\mathsf{J}_{e_2\times e_1} & \vline & \dfrac{1}{\rho}\mathsf{J}_{e_2\times e_2} \\
	\end{array}
	\right),
	\end{equation}
	where $\mathsf{J}$ is all one matrix of suitable order.
	Then,
	$$\mathsf{A}^{-1}=\mathsf{D}^{-1}\mathsf{A}^T(\mathsf{I_v}+\mathsf{R})^{-1}=\mathsf{D}^{-1}\mathsf{A}^T\left(\mathsf{I_v}-\dfrac{\rho }{\lambda(\rho +1)^2}\mathsf{R} \right).$$
\end{thm}
\begin{thm}\label{thm:necessary condition}
	Let $\mathcal{D}$ be a Ryser design of order $v$ and index $\lambda$ with replication numbers $r_1\text{ and }r_2$. Let $r=r_1-r_2$. 
	\begin{enumerate}
		\item If  $\mathcal{D}$ is a Ryser design of Type-1 with $D=0 ~~\text{, then }~~ v= 2\lambda\pm\sqrt{(2\lambda-1)^2+(r-1)^2-1}$  and $\{(2\lambda-1)^2+r(r-2)\}$ is a perfect square.
		\item 	If  $\mathcal{D}$ is a Ryser design of Type-1 with $D=-1 ~~\text{, then }~~v=2\lambda\pm\sqrt{(2\lambda-1)^2+(r-1)^2+4r-1}$ and $\{(2\lambda-1)^2+r(r+2)\}$ is a perfect square.
	\end{enumerate}
\end{thm}

\begin{thm}\label{thm:result3}(cf. conjecture \ref{con:bound_on_v})
	Let $\mathcal{D}$ be a Ryser design of order $v$ and index $\lambda$. 
	\begin{enumerate}[(a)]
		\item If $D\leq -1$, then $v\geq 4\lambda-1$.
		\item If $D\geq 0$, then  $\lambda^2 +\lambda+1 \geq v$.
	\end{enumerate}
\end{thm}
\noindent
Finally, we discuss a special case of Ryser designs.
\begin{thm}\label{thm:result4}
	Let $\mathcal{D}$ be a Ryser design of order $v$, index $\lambda$ and replication numbers $r_1$ and $r_2$ and two block sizes $k_1>k_2$. 
	\begin{enumerate}[(a)]
		\item If $k_1=2\lambda+t_1a$ with $2t_1c+\lambda>e_1$ that is $2t_1>x$, then $\mathcal{D}$ is of Type-1. \\Or
		\item If $k_2=2\lambda-t_2a$ with $2t_2d+\lambda>e_2$ that is $2t_2>y$, then $\mathcal{D}$ is of Type-1.
	\end{enumerate}
\end{thm}

\section{The inverse of incidence matrix}
\noindent
We begin by stating a well known and an important  relation of the incidence matrix $\mathsf{A}$ of a Ryser design and the diagonal matrix with diagonal entries $k_i-\lambda$, where $~k_i,~~ i=1,2,\ldots,v~$ is the $i$-th column sum (block size) of the incidence matrix $\mathsf{A}$. 
\begin{lemma}\label{lemma:AD^-1A^T}
	Let $\mathsf{A}$ be the incidence matrix of a Ryser design  with index $\lambda$ and block sizes $k_i,~~ i=1,2,\ldots,v$. 
	Then, $\mathsf{A}^T\mathsf{A}= \mathsf{D}+\lambda \mathsf{J_v} \text{ and }
	\mathsf{A}\mathsf{D}^{-1}\mathsf{A}^T=\mathsf{I_v}+\mathsf{R},$ where $\mathsf{D}~~\text{ and }~~\mathsf{R}$ are as defined in Equation (\ref{MatrixD}) (\ref{MatrixR}) respectively and $\mathsf{J}$ is all one matrix.
\end{lemma}

\noindent
This result as also Equation (\ref{sum_of_all}) of previous section are a consequence of the following results of Ryser \cite{Ry2}: 

\begin{lemma}
	{	Let $\,\mathsf{X}=[x_{ij}] \,\mbox{ and }\, \mathsf{Y}=[y_{ij}]\,$  be real matrices of order$~ v  \text{ that satisfy the matrix}$ $\text{equation } \,\mathsf{X}\mathsf{Y}=\mathsf{D}+[\sqrt{\lambda_i\lambda_j}]\,, \text{where} ~\mathsf{D}~ \text{is the diagonal matrix}$  $\,diag[k_1-\lambda_1,k_2-\lambda_2,\ldots, k_v-\lambda_v] \,$ and the scalars  $k_i-\lambda_i\,$  and $ \lambda_j $ are positive and non-negative respectively. Then  $\mathsf{Y}\mathsf{D}^{-1}\mathsf{X}=\mathsf{I}+t[y_ix_j], \text{ where } \mathsf{D}^{-1}$ denotes the inverse of $\mathsf{D},~ \mathsf{I}$ denotes the identity matrix of order $v$ and the scalars $\,t,\, y_i,\, x_j \,$ are determined by the equations,\\
		$t=1+ \left(\dfrac{\lambda_1}{k_1-\lambda_1}\right) +\cdots +\left(\dfrac{\lambda_n}{k_v-\lambda_v}\right)$\\
		$ty_i=\left(\dfrac{\sqrt{\lambda_1}}{k_1-\lambda_1}\right)y_{i1} +\cdots +\left(\dfrac{\sqrt{\lambda_n}}{k_v-\lambda_v}\right)y_{iv}$\\
		$tx_j=\left(\dfrac{\sqrt{\lambda_1}}{k_1-\lambda_1}\right)x_{1j} +\cdots +\left(\dfrac{\sqrt{\lambda_n}}{k_v-\lambda_v}\right)x_{vj}$.}
\end{lemma}
\begin{defn}\rm{
		Let $\mathsf{A}$ be a $(0,1)$-matrix of order $m$ by $v>3$ that satisfies the matrix equation $\mathsf{A}^T\mathsf{A}=\mathsf{D}+[\sqrt{\lambda_i\lambda_j}]$, where $\mathsf{A}^T$ denotes transpose of $\mathsf{A}$ and $\mathsf{D}$ denotes the diagonal matrix  $\mathsf{D}=diag(k_1-\lambda_1,k_2-\lambda_2,\ldots,k_v-\lambda_v)$
		with $k_i-\lambda_i$ and $\lambda_i$ positive and also Fisher Type inequality implies $m\geq v$. We call a configuration whose
		incidence matrix $\mathsf{A}$ fulfills these requirements a \textit{multiplicative design} on
		the parameters $k_1,k_2,\ldots,k_v$ and $\lambda_1,\lambda_2,\ldots, \lambda_v$.}
\end{defn}
\begin{corollary}\label{cor_1}
	Let $\mathsf{A}$ be the incidence matrix of a multiplicative design on the parameters $k_1,k_2,\ldots,k_v$ and $\lambda_1,\lambda_2,\ldots, \lambda_v$. Then,\\
	\begin{equation*}
	\mathsf{A}\mathsf{D}^{-1}\mathsf{A}^T=\mathsf{I_v}+t[x_ix_j] 
	\end{equation*}
	where 
	\begin{align}
	t&=1+ \left(\dfrac{\lambda_1}{k_1-\lambda_1}\right) +\cdots +\left(\dfrac{\lambda_v}{k_v-\lambda_v}\right)\\
	tx_i&=\left(\dfrac{\sqrt{\lambda_1}}{k_1-\lambda_1}\right)a_{i1} +\cdots +\left(\dfrac{\sqrt{\lambda_v}}{k_v-\lambda_v}\right)a_{iv}.\label{txi_formula}
	\end{align}
\end{corollary}
\begin{corollary}\label{cor_2}
	The parameters $~~k_1,k_2,\ldots,k_v~~$ and $~~\lambda_1,\lambda_2,\ldots ,\lambda_v~~$ of a multiplicative design satisfy 
	$$\left[\dfrac{k_1^2}{k_1-\lambda_1} +\cdots +\dfrac{k_v^2}{k_v-\lambda_v}-v \right]\left[ 1+ \dfrac{\lambda_1}{k_1-\lambda_1} +\cdots +\dfrac{\lambda_v}{k_v-\lambda_v}\right]=\left[\dfrac{\sqrt{\lambda_1}}{k_1-\lambda_1}k_1 +\cdots +\dfrac{\sqrt{\lambda_v}}{k_v-\lambda_v}k_v\right]^2$$.
\end{corollary}
\noindent
Note that if we set $\,k_1=k_2=\cdots=k_v=k\,\text{ and }\, \lambda_1=\lambda_2=\cdots=\lambda_v=\lambda$, 
then Equation (\ref{txi_formula}) reduces to $\,k-\lambda=k^2-kv\,$ for symmetric block design. Further we get a Ryser design if we set $\, \lambda_1=\lambda_2=\cdots=\lambda_v=\lambda\,$  with at least two different block sizes.\\
We state the following result from Miller \cite{miller} which will be used to prove Theorem \ref{Inverse of incidence matrix}.
\begin{thm}\label{thm:millerMatrixInverseThm}
	Let $\mathsf{G}$ and $\mathsf{H}$ be arbitrary square matrices of the same order. If $\mathsf{G}$ is non singular and $\mathsf{H}$ has rank one, then $(\mathsf{G}+\mathsf{H})^{-1}=\mathsf{G}^{-1}-\dfrac{1}{1+g}\mathsf{G}^{-1}\mathsf{H}\mathsf{G}^{-1}$, where $g=tr \mathsf{H}\mathsf{G}^{-1}$.
\end{thm}
\begin{proof}[\textit{Proof of Theorem \ref{Inverse of incidence matrix}:}]
	By definition of a Ryser design we know that $\mathsf{A}^T\mathsf{A}= \mathsf{D}+\lambda \mathsf{J_v},  \text{ where }  \mathsf{D}$ is as defined in Equation (\ref{MatrixD}) and  $\mathsf{J_v}$ is all 1 matrix.
	Then by simple manipulations we get 
	{\scriptsize\small\begin{align*}
		det(\mathsf{A}^T\mathsf{A})&=\left[1+\lambda\left(\dfrac{1}{(k_1-\lambda)} +\cdots +\dfrac{1}{(k_v-\lambda)}   \right)\right](k_1-\lambda)\cdots(k_v-\lambda)\\
		&= \left[1+\lambda\sum\limits_{j=1}^v\dfrac{1}{(k_j-\lambda)} \right]\prod\limits_{i=1}^v(k_i-\lambda)\\
		\intertext{Equation (\ref{sum_of_all}) implies,}
		det(\mathsf{A}^T\mathsf{A})&= \left[1+\lambda\left(\dfrac{(\rho +1)^2}{\rho}-\dfrac{1}{\lambda}\right) \right]\prod\limits_{i=1}^v(k_i-\lambda)\\
		&= \left[\lambda\dfrac{(\rho +1)^2}{\rho} \right] \prod\limits_{i=1}^v(k_i-\lambda)\neq 0.
		\end{align*}}
	Hence$\,\,\, \mathsf{A}\,\,\,$is invertible.\\
	\noindent
	By Lemma \ref{lemma:AD^-1A^T} we have $\mathsf{A}\mathsf{D}^{-1}\mathsf{A}^T=\mathsf{I_v}+\mathsf{R} $, where $\mathsf{R}$ is as defined in Equation (\ref{MatrixR}).  
	As $\mathsf{A}\mathsf{D}^{-1}\mathsf{A}^T$ is invertible, (since $\mathsf{A} $ and $\mathsf{D} $ are) so is $\mathsf{I_v}+\mathsf{R}$. Also note that $\mathsf{R}$ is symmetric and has rank one.
	Now, $~(\mathsf{A}\mathsf{D}^{-1}\mathsf{A}^T)^{-1}=(\mathsf{I_v}+\mathsf{R})^{-1} \,$ gives $\,(\mathsf{A}^T)^{-1}\mathsf{D}\mathsf{A}^{-1}=(\mathsf{I_v}+\mathsf{R})^{-1}$ and hence $\mathsf{A}^{-1}=\mathsf{D}^{-1}\mathsf{A}^T(\mathsf{I_v}+\mathsf{R})^{-1}$. 
	Use of Theorem \ref{thm:millerMatrixInverseThm} will now be made to obtain the inverse of $\mathsf{I}_v+\mathsf{R} $. \\\\
	The trace of $\mathsf{R}$ is easily seen to be $~e_1\rho+\dfrac{e_2}{\rho}$. Therefore,	
	\begin{align*}	
	(\mathsf{I}_v+\mathsf{R})^{-1}&=\mathsf{I}_v - \dfrac{1}{1+e_1\rho+\dfrac{e_2}{\rho}}\mathsf{R}\\
	&=\mathsf{I}_v - \dfrac{\rho}{\rho + \rho^2 e_1+e_2}\mathsf{R}.
	\intertext{By Equation (\ref{e1rhopluse2byrho_this_ch})}
	1 +\rho e_1+\dfrac{e_2}{\rho}&= \lambda\dfrac{(\rho +1)^2}{\rho}.\\
	\intertext{Hence we have,} 
	(\mathsf{I}+\mathsf{R})^{-1}&=\mathsf{I}_v-\dfrac{\rho }{\lambda(\rho +1)^2}\mathsf{R} .
	\end{align*}
\end{proof}

\section{A necessary condition for a Ryser design to be of Type-1}
\noindent

\begin{proof}[\textit{Proof of Theorem \ref{thm:necessary condition}:}]
	Since $r=r_1-r_2$ we have $(\rho+1)/(\rho-1)=(v-1)/r $
	which implies
	\begin{equation}\label{rho}
	\rho= \dfrac{v-1+r}{v-1-r}
	\end{equation}
	Equations (\ref{e1Dform}) and (\ref{e2Dform}) imply
	$v=e_1+e_2=\lambda +\left(\dfrac{\lambda+D}{\rho}\right) + \lambda+[\lambda-(D+1)]\rho~~$ which on simplification gives
	$[\lambda-(D+1)]\rho^2-(v-2\lambda)\rho +(\lambda+D)=0$.
	Hence we get,
	\begin{equation}\label{root}
	\rho= \dfrac{(v-2\lambda)\pm\sqrt{(v-2\lambda)^2-4[\lambda-(D+1)](\lambda+D)}}{2[\lambda-(D+1)]}.
	\end{equation}
	Now, Equations (\ref{rho}) and (\ref{root}) imply\\
	$$\dfrac{(v-2\lambda)\pm\sqrt{(v-2\lambda)^2-4[\lambda-(D+1)](\lambda+D)}}{2[\lambda-(D+1)]}=\dfrac{v-1+r}{v-1-r} $$
	which on simplification gives,\\ 
	$ v^3-v^2[4\lambda+1]+v[8\lambda+4rD-(r-1)^2]-[4\lambda +4rD -(r-1)^2] = 0. $\\
	Let $f(v)=v^3-v^2[4\lambda+1]+v[8\lambda+4rD-(r-1)^2]-[4\lambda +4rD -(r-1)^2]$.\\ Then,
	$f(1)=1-[4\lambda+1]+[8\lambda+4rD-(r-1)^2]-[4\lambda +4rD -(r-1)^2]=0.$\\
	After factorization we get $f(v) = (v-1)[v^2-4v\lambda +4\lambda+4Dr-(r-1)^2]$.\\
	Since $\,\,v \neq 1 ~~\text{ and }~~ f(v)=0, ~~\text{ we have }~~ v^2-4v\lambda +4\lambda+4Dr-(r-1)^2=0$.\\
	This implies $~~ v= 2\lambda\pm\sqrt{(2\lambda-1)^2+(r-1)^2-4Dr-1}$.\\
	By Theorem \ref{thm:D=0D=-1} a Ryser design is of Type-1 if and only if $D=0~~ \text{ or } D=-1$.
	Now $D=0 ~~\text{ gives }~~ v= 2\lambda\pm\sqrt{(2\lambda-1)^2+(r-1)^2-1}$,  then $\{(2\lambda-1)^2+r(r-2)\}$ is a perfect square.
	If there exists a Ryser design of Type-1 with $D=-1$, then we get
	$v=2\lambda\pm\sqrt{(2\lambda-1)^2+(r-1)^2+4r-1}$ which implies $\{(2\lambda-1)^2+r(r+2)\}$ is a perfect square.
\end{proof}
\begin{corollary}
	Let $\mathcal{D}$ be a Ryser design of order $v$ and index $\lambda$. Then $v\geq 4\lambda-1$ if and only if $e_2-e_1\geq2D+1$.
\end{corollary}
\begin{proof}[Proof:]
	We know that $v= 2\lambda\pm\sqrt{(2\lambda-1)^2+(r-1)^2-4Dr-1}$, where $r=r_1-r_2$. This on simplification gives
	$(v-4\lambda+1)^2(v-1)=r(r-2-4D)$. Now, $v\geq4\lambda-1$ if and only if $r-2-4D\geq0$ if and only if $e_2-e_1\geq2D+1$.
\end{proof}

\begin{proposition}
	Let $\mathcal{D}$ be a Ryser design of order $v$ and index $\lambda$. Let $A$ be a large block and $B$ be a small block of $\mathcal{D}$. Then $ \tau_1(A)-1 \geq D\geq-\tau_2(B) $. In general $\lambda-1 > D>-\lambda$.
\end{proposition}
\begin{proof}[Proof:]
	Let $B$ be any block with $|B|=\tau_1(B)+\tau_2(B)$. By Equations (\ref{e1form}) and (\ref{tau1rhotau2form}), we get
	$\rho(e_1-\tau_1(B))=\tau_2(B)+D$
	which implies $\tau_2(B)+D \geq 0$.
	In particular, if $B$ is a small block, then $D\geq-\tau_2(B)$. 
	Let $A$ be any block with $|A|=\tau_1(A)+\tau_2(A)$. By Equations (\ref{e2form}) and (\ref{tau1rhotau2form}), we get
	$\rho [\tau_1 -(D+1)] = e_2-\tau_2$ which implies $\tau_1(A)-(D+1) \geq 0$. In particular, if $A$ is a large block, then $\tau_1(A)-1\geq D$.
	By Lemma \ref{lemma:blocksize} we have $\lambda >\tau_1(A)$ and $\lambda>\tau_2(B)$ we get,\\
	\begin{equation}\label{bound_on_D_interms_of_lambda}
	\lambda-1 > D>-\lambda.
	\end{equation}
	This completes the proof.
\end{proof}

\begin{proof}[\textit{Proof of Theorem \ref{thm:result3}:}]
	By Theorem \ref{thm:necessary condition}$~~v= 2\lambda\pm\sqrt{(2\lambda-1)^2+(r-1)^2-4Dr-1}.$ This implies
	$(v-2\lambda)^2-(2\lambda-1)^2 = (r-1)^2-4Dr-1. \text{ If } D\leq -1\text{, then } (r-1)^2-4Dr-1 \geq 0\,\,\text{ which implies } (v-2\lambda)^2-(2\lambda-1)^2\geq 0 \text{ and hence } v\geq4\lambda-1 $. Thus $D\leq -1$ implies $v\geq 4\lambda-1$. Using Theorem \ref{thm:rho<lambda}$~~$ and Equation (\ref{root}),  $$\lambda \geq   \dfrac{(v-2\lambda)+\sqrt{(v-2\lambda)^2-4[\lambda-(D+1)](\lambda+D)}}{2[\lambda-(D+1)]}.$$
	By Equation (\ref{bound_on_D_interms_of_lambda}) we have $\,\lambda-1 >D$. If \,\,$D>0\,\,$, then we get 
	$$ 2\lambda[\lambda-(D+1)] \geq (v-2\lambda)+\sqrt{(v-2\lambda)^2-4[\lambda-(D+1)](\lambda+D)}$$ which on simplification gives $\lambda^2+\lambda+1 -D[\lambda-(1/\lambda)] \geq v$. Now if $D=0$ we get $\lambda^2 +\lambda+1 \geq v$ and hence if $D\geq 0 \text{ we have }\lambda^2 +\lambda+1 \geq \lambda^2+\lambda+1 - D[\lambda-(1/\lambda)]\geq v$.
\end{proof}
\noindent

\section{Special Ryser designs with two block sizes}

\noindent

\begin{thm}\label{Thm:unique_block}
	Let $\mathcal{D}$ be a Ryser design of Type-2 of order $v$, index $\lambda$ and replication numbers $r_1$ and $r_2$. 
	\begin{enumerate}[(a)]
		\item If there exists a block $A$ of size $k=2\lambda+ta$, where $2tc+\lambda>e_1$ that is $t>x/2$, then $A$ is the unique block of size $k=2\lambda+sa$ with $s>x/2$.
		\item If there exists a block $B$ of size $k=2\lambda-ta$, where $2td+\lambda>e_2$ that is $t>y/2$, then $B$ is the unique block of size $k=2\lambda-sa$ with $s>y/2$.
	\end{enumerate}
\end{thm}
\begin{proof}[Proof:]
	We give a proof of (a). The proof of (b) is similar.
	By Proposition \ref{prop:complement-properties} we know that, if we complement $\mathcal{D}$ with respect to a block $B$ with $|B|=k$, then we get a new Ryser design $\overline{\mathcal{D}}$, with index $\overline{\lambda} =k-\lambda$ and the same replication numbers.
	Then $r_1=2\lambda+xa=2\overline{\lambda}+\overline{x}a$ implies $ 2(2\lambda-k)=(\overline{x}-x)a $.
	Now, if $k>2\lambda$, then $k=2\lambda+ta$ which gives us
	$\overline{x}=x-2t$.
	Therefore $\overline{x}>0$ if and only if $t<x/2$.
	By Equation (\ref{e2lambdaxcform})
	we have $\overline{e_2}=\overline{\lambda}+\overline{x}c$.
	Hence if $\overline{e_2}>\overline{\lambda}$, then $\overline{x}>0$.\\
	Let $A$ be a block of size $k=2\lambda+ta$, where $2tc+\lambda>e_1$ that is $t>x/2$. Let $A'$ be a block of size $k'=2\lambda+sa$ with $s>x/2$. Then we claim that $A'=A$.
	We can then choose $A'$ so that $s$ is the smallest with that property.
	Then in the new design $\overline{\mathcal{D}}$ obtained by complementing $\mathcal{D}$ with respect to the block $A'$ we have $\overline{e_2}<\overline{\lambda}$ and hence it can not have large or average blocks. But blocks of size $\geq k'$ (in $\mathcal{D}$) become large or average blocks in $\overline{\mathcal{D}}$. In particular, $A$ becomes average or large in $\overline{\mathcal{D}}$ which is a contradiction. This contradiction proves that $A'=A$. Hence the claim.
\end{proof}

\begin{proof}[\textit{Proof of Theorem \ref{thm:result4}:}]
	Clearly we can not have $k_1=2\lambda+t_1a$ with $2t_1>x$ and  $k_2=2\lambda-t_2a$ with $2t_2>y$ for in that case  by Theorem \ref{Thm:unique_block} the design will have only two blocks. Hence precisely one of (a) or (b) must occur. Without loss of generality let $\mathcal{D}$ be a Ryser design with two block sizes $k_1>k_2$, where $k_1=2\lambda+t_1a$ with $2t_1>x$.
	Then by Theorem \ref{Thm:unique_block} $\mathcal{D}$ has a unique block of size $k_1$ and hence all the remaining $v-1$ blocks are of size $k_2$. Now by Theorem \ref*{Thm:Type-1_iff_2columnsum_1occurs_exactly_once} $\mathcal{D}$ is of Type-1. The other case is similar.
\end{proof}

%
\end{document}